\documentclass[11pt,reqno]{amsart}
\usepackage{amsmath,amsthm,amscd,amsfonts,amssymb,color}
\usepackage{cite}
\usepackage{graphicx}
\usepackage{xcolor}
\usepackage[mathscr]{eucal}
\usepackage{upgreek}
\usepackage[bookmarksnumbered,colorlinks,plainpages]{hyperref}
\newtheorem{theorem}{Theorem}[section]

\newtheorem*{Acknowledgement}{\textnormal{\textbf{Acknowledgement}}}

\newtheorem{proposition}[theorem]{Proposition}
\newtheorem{corollary}[theorem]{Corollary}

\theoremstyle{definition}

\newtheorem{definition}[theorem]{Definition}
\newtheorem{example}[theorem]{Example}

\newtheorem{Open Prob}[theorem]{Open Problem}

\theoremstyle{remark}
\numberwithin{equation}{section}
\def\DJ{\leavevmode\setbox0=\hbox{D}\kern0pt\rlap{\kern.04em\raise.188\ht0\hbox{-}}D}
\begin{document}

\title[Perimetric contraction on quadrilaterals]{ Perimetric contraction on quadrilaterals and related fixed point results }


\author[A.\ Banerjee, P.\ Mondal, L.K.\ Dey]
{Anish Banerjee$^{1}$, Pratikshan Mondal$^{2}$, Lakshmi Kanta Dey$^{1}$}

\address{{$^{1}$}   Anish Banerjee,
                    Department of Mathematics,
                    National Institute of Technology
                    Durgapur, Durgapur,  India}
                    \email{anishbanerjee1997@gmail.com}
\address{{$^{2}$}   Pratikshan Mondal,
                    Department of Mathematics,
                    Durgapur Government College,
                    Durgapur, India}
                    \email{real.analysis77@gmail.com}
                    
\address{{$^{1}$}   Lakshmi Kanta Dey,
                    Department of Mathematics,
                    National Institute of Technology
                    Durgapur, Durgapur,  India}
                    \email{lakshmikdey@yahoo.co.in}


\begin{abstract}
In this article, we introduce a four-point analogue of Banach-type, Kannan-type, and Chatterjea-type contractions, and examine their properties. We establish sufficient conditions under which these mappings achieve fixed points in a complete metric space. Notably, the classical Banach contraction principle emerges as a special case of our results. To illustrate our theoretical findings, we present several non-trivial examples.

\end{abstract}

\maketitle

\setcounter{page}{1}


\section{Introduction and Preliminaries}
Fixed point theory plays a crucial role in mathematics, where many problems can be framed as fixed point problems. These problems involve investigating the existence and uniqueness of solutions. Applications of fixed point theory span diverse areas, including matrix equations, differential equations, integral equations, optimization, and machine learning.

The foundational work in this field dates back to Stefan Banach's introduction of the Banach contraction principle \cite{B22} in 1922. This principle guarantees the existence and uniqueness of fixed points of contraction mappings in a complete metric space. Subsequently, other prominent researchers contributed significantly to the evolution of fixed point theory. As a result, the concept of Banach contraction has been extended in various ways by relaxing the contraction condition and considering different topologies.


There are various classical results in the literature of fixed point theory. These results generalize Banach's theorem in various ways. Nadler \cite{N69} extended Banach's theorem from single-valued mapping to multi-valued mappings. Kirk \cite{K65} has proposed the nonexpansive mapping type extension of the Banach contraction principle. Berinde \cite{BP20} has introduced the enriched contractions generalizing the contraction mappings in normed linear spaces. Generalizing the underlying space, Browder \cite{B68} has initiated the fixed point result in topological vector spaces. Wardowski \cite{W12} has disclosed $F$-contractions using implicit functions extending the contraction mappings. 
Khojasteh, et al. induced the idea of $Z$-contraction mappings by utilizing simulation functions, see \cite{KSR15},\cite{SDCH19}. Kannan \cite{RK68} obtained a fixed point result for a class of mappings which characterizes the completeness of the metric space. Chatterjea \cite{C72} proposed a class of mappings independent of Banach's class and identified the prerequisites for reaching fixed point.  
Now, we recall some well-known results. 




\medskip

Recently in $2023$, Petrov \cite{P23} introduced the notion of a new class of mappings that can be distinguished as the mapping that contract the perimeters of triangles and proved a fixed point result. Let us recall that.

\begin{definition}(\cite{P23})
Let $(Y,\rho)$ be a metric space with at least three points. Then the mapping $T: Y \to Y$ is defined as contracting perimeters of triangles if there is an $\alpha \in [0,1)$ such that 
\begin{equation}
\rho(Tp,Tq)+ \rho(Tq,Tr)+ \rho(Tr,Tp) \le \alpha (\rho(p,q)+ \rho(q,r)+ \rho(r,p)) \label{cpt}
\end{equation}
for three pairwise distinct points $p,q,r \in Y$.
\end{definition}

These mappings can attain fixed points in a complete metric space if and only if it has no periodic points of prime period $2$. There are at most two fixed points. 


\medskip

In 2024, Petrov along with Bisht introduced the three-point analogue of both Kannan type mappings \cite{PB24} and Chatterjea-type mappings \cite{BP24} utilizing the notion of mapping contracting perimeters of triangles and developed fixed point results.

\begin{definition}(\cite{PB24})
Let $(Y,\rho)$ be a metric space with at least three points. Then $T: Y \to Y$ is a generalized Kannan type mapping on $Y$ if there is a $0 \le \delta < \frac{2}{3}$ such that 
\begin{equation}
\rho(Tp,Tq)+ \rho(Tq,Tr)+ \rho(Tr,Tp) \le \delta (\rho(p,Tp)+ \rho(q,Tq)+ \rho(r,Tr)) \label{gkc}
\end{equation} 
for any three pairwise distinct points $p,q,r \in Y$.
\end{definition}



\begin{definition}(\cite{BP24})
Let $(Y,\rho)$ be a metric space with at least three points. Then $T: Y \to Y$ is a generalized Chatterjea type mapping on $Y$ if there is $0 \le \lambda < \frac{1}{3}$ such that 
\begin{align}
\rho(Tp,Tq)+ \rho(Tq,Tr)+ \rho(Tr,Tp) &\le \lambda (\rho(p,Tq)+ \rho(p,Tr)+\rho(q,Tp) \notag\\
&\hspace{1cm}+ d(q,Tr)+ d(r,Tp)+ d(r,Tq))\label{gcc}
\end{align}
for any three pairwise distinct points $p,q,r \in Y$
\end{definition}

In a complete metric space, both a generalized Kannan type mapping and a generalized Chatterjea type mapping attain fixed points if it does not achieve periodic points of prime period $2$. There are at most two fixed points.


\medskip

We find the results intriguing and wish to explore the four-point analogue of previous findings. Our goal is to establish adequate conditions ensuring the existence and uniqueness of fixed points. Additionally, we aim to compare these different classes of mappings and uncover any relationships between them.

\medskip

In the third section, we delve into a novel type of mapping characterized by mapping that contracts the perimeter of quadrilaterals. We prove a fixed point result of such mapping in a complete metric space. Notably, achieving a fixed point requires avoiding periodic points of prime period 2 and 3. Interestingly, the class of contraction mappings is a subset of these perimeter-based mappings. As a straightforward consequence, we recover Banach’s fixed point theorem. To validate our results, we provide illustrative examples.



\medskip

In the fourth section, we introduce Kannan type perimetric contraction on quadrilaterals and establish a fixed point result for these mappings. We derive a necessary condition for the fixed point to be unique. Additionally, we investigate the relationship between the class of Kannan type perimetric contraction on quadrilaterals and both generalized Kannan type mappings and mappings that contract the perimeter of quadrilaterals. Notably, we demonstrate that these classes are independent. To illustrate our findings, we provide non-trivial examples. 


\medskip

In the fifth section, we introduce Chatterjea type perimetric contraction on quadrilaterals and obtain a fixed point result for these mappings. The connection between the classes of Chatterjea type perimetric contraction on quadrilaterals with generalized Chatterjea type mappings and mapping contracting perimeter of quadrilaterals has been performed. An adequate condition has been implemented for the uniqueness of the fixed point.

Throughout the paper, we have denoted $(M,d)$ as metric space, 
$|M|$ as the cardinality of the set $M$, 
$\mathbb{N}$ as the set of natural numbers, 
and $Fix(T)$ as the collection of all fixed points of $T$.


The concept of a periodic point is defined as follows:

Let $T$ be a mapping on the metric space $M$. A point $m \in M$ is said to be a periodic point of period $p$ if $T^p{m}=m$. The prime period of $m$ is the least positive integer $p$ for which $T^p{m}=m$.

\section{Perimetric contraction }
We begin the section with the following definition of perimetric contraction on quadrilaterals:

\begin{definition}\label{D31}
Let $(M,d)$ be a metric space with at least four points. Then the mapping $T:M \to M$ is said to be a perimetric contraction on quadrilaterals in $M$ if there is an $\alpha \in [0,1)$ such that 
\begin{align}
&d(Tp,Tq)+ d(Tq,Tr)+ d(Tr,Ts)+ d(Ts,Tp)\notag\\
&\le \alpha (d(p,q)+ d(q,r)+ d(r,s)+ d(s,p)) \label{pcq}
\end{align}
for all distinct points $p,q,r,s \in M$.
\end{definition}
Now, we investigate the continuity of these mappings.
\begin{theorem}\label{T32}
A perimetric contraction on quadrilaterals is continuous.
\end{theorem}

\begin{proof}
Let $(M,d)$ be a metric space with at least four points, and let $T: M \to M$ be a perimetric contraction on quadrilaterals. Let $m'\in M$ be arbitrary. If $m'$ is an isolated point of $M$, then it is obvious that $T$ is continuous at $m'$. 

Now, suppose that $m'$ is a limit point of $M$. Therefore, it remains to show that for any $\epsilon>0$, there exists a $\delta>0$ such that $d(Tm,Tm')<\varepsilon$ for all $m \in M$ satisfying $d(m,m')<\delta$. 

Let $\varepsilon>0$ be arbitrary. Choose $\delta>0$ be such that $0<\delta<\frac{\varepsilon}{6\alpha}$.

Since $m'$ is a limit point of $M$, there exist $a,b \in M$ with $a\neq b\neq m'$ such that $d(m',a)<\delta$ and $d(m',b)<\delta$. Now, for all $m \in M$ with $m \neq m'$ satisfying $d(m, m')<\delta$, we have
\begin{align*}
d(Tm,Tm') &\le d(Tm,Tm')+ d(Tm',Ta)+ d(Ta,Tb)+ d(Tb,Tm)\\
&\le \alpha(d(m,m')+ d(m',a)+ d(a,b)+ d(b,m))\\
&\le 2\alpha(d(m,m')+ d(m',a)+ d(m',b))\\
&<6\alpha\delta<\varepsilon
\end{align*}
and hence the result follows.
\end{proof}
Now, we are ready to establish a condition that is both necessary and sufficient for the existence of fixed point(s) for perimetric contraction on quadrilaterals.
\begin{theorem}\label{T33}
Let us suppose a complete metric space $(M,d)$ with at least four points. Consider a mapping $T: M \to M$ to be a perimetric contraction on quadrilaterals in $M$. Then, $T$ attains a fixed point in $M$ if and only if it does not attain periodic points of prime period 2 and 3. Furthermore, $T$ can attain at most three fixed points.
\end{theorem}

\begin{proof}
Let $T: M \to M$ be a perimetric contraction on quadrilaterals in $M$ and let $T$ does not attain periodic points of prime period $2$ and $3$. 

Let $a_0 \in M$ be chosen arbitrarily. Define $Ta_0=a_1, Ta_1=a_2, \cdots, Ta_n=a_{n+1}, \cdots$. If $a_n$ is a fixed point of $T$ for any $n$, then we are done.

Now, assume that $a_n$ is not a fixed point of $T$ for all $n$. Since $a_n$ is not a fixed point of $T$, it follows that $a_0 \neq a_{1}, a_{1} \neq a_{2}$ and so on. Since $T$ does not attain periodic points of prime period $2$, then $a_0 \neq a_{2}, a_{1} \neq a_{3}$ and so on. Again, since $T$ does not attain periodic points of prime period $3$, we have $a_0 \neq a_{3}$ and so on. Therefore, any four consecutive elements of $\{a_n\}$ are distinct. 

Let $\lambda_n = d(a_n,a_{n+1})+ d(a_{n+1},a_{n+2})+ d(a_{n+2},a_{n+3})+ d(a_{n+3},a_n)$ for all $n \in \mathbb{N} \cup \{0\}$ so that $\lambda_n>0$ for all $n\in \mathbb{N} \cup \{0\}$.


Now, for any $n \in \mathbb{N} \cup \{0\}$, we have $\lambda_n \le \alpha \lambda_{n-1}$.
Then, it is clear that
\begin{align*}
d(a_0,a_1) &\le \lambda_0,\\
d(a_1,a_2) &\le \lambda_1 \le \alpha \lambda_0,\\
\vdots\\
d(a_n,a_{n+1}) &\le \lambda_n \le \alpha^n \lambda_0.
\end{align*}

Now, for all $n\in \mathbb{N}\cup \{0\}$ and for any $p=1,2,3,\cdots$, we have
\begin{align*}
d(a_n,a_{n+p}) &\le d(a_n,a_{n+1})+ d(a_{n+1},a_{n+2})+ \dots + d(a_{n+p-1},a_{n+p})\\
&\le \alpha^n \lambda_0+ \alpha^{n+1} \lambda_0+ \dots + \alpha^{n+p-1} \lambda_0 \\
&\le \alpha^n (1+ \alpha+ \dots + \alpha^{p-1}) \lambda_0\\
&\le \alpha^n \dfrac{1-\alpha^p}{1-\alpha} \lambda_0.
\end{align*}

This implies that $d(a_n,a_{n+p}) \to 0$ as $n \to \infty$ and for any $p=1,2,3,\cdots$. Hence, $\{a_n\}$ is a Cauchy sequence in $M$ and therefore convergent, as $M$ is complete. Let $a_n\to a^* \in M$. Now,
\begin{align*}
&d(a^*,Ta^*)\\
&\le d(a^*,a_n)+ d(a_n,Ta^*)\\
&\le d(a^*,a_n)+ d(Ta_{n-1},Ta^*)+ d(Ta_{n-2},Ta_{n-1})+ d(Ta_{n-1},Ta_n)+ d(Ta_n,Ta^*)\\
&\le d(a^*,a_n)+ \alpha (d(a_{n-1},a^*)+ d(a_{n-2},a_{n-1})+ d(a_{n-1},a_n)+ d(a_n,a^*)).
\end{align*}

Taking $n \to \infty$, we get $Ta^*=a^*$, and $a^* \in Fix(T)$.

Conversely, let $T$ have a fixed point say, $p \in M$. Suppose that $T$ attain a periodic point $q$ of prime period $2$ and a periodic point $r$ of prime period $3$.

Let, $Tq=a, Tr=b, Tb=c$. Then using (\ref{pcq}), we have
\begin{align}
& d(q,p)+ d(p,a)+ d(a,b)+ d(b,q) \le \alpha (d(a,p)+ d(p,q)+ d(q,r)+ d(r,a)),\label{331}\\
& d(q,p)+ d(p,a)+ d(a,c)+ d(c,q) \le \alpha (d(a,p)+ d(p,q)+ d(q,b)+ d(b,a)),\label{332}\\
&\hspace{5cm}\text{and} \notag\\
& d(q,p)+ d(a,p)+ d(a,r)+ d(r,q) \le \alpha (d(a,p)+ d(p,q)+ d(q,c)+ d(c,a)).\label{333}
\end{align}

Adding (\ref{331}), (\ref{332}) and (\ref{333}), we get $\alpha\ge 1$ which is a contradiction to (\ref{pcq}).

Thus, $T$ cannot attain periodic points of prime period $2$ and $3$.

Let us suppose that $T$ has four distinct fixed points, say, $p,q,r,s$.
Then
\begin{align*}
&d(Tp,Tq)+ d(Tq,Tr)+ d(Tr,Ts)+ d(Ts,Tp)\\
&\le \alpha (d(p,q)+ d(q,r)+ d(r,s)+ d(s,p))
\end{align*}
which again implies that $\alpha\ge 1$ - a contradiction to (\ref{pcq}). Hence, the result follows. 
\end{proof}

Below, we present the following examples in support of Theorem~\ref{T33}. The first one is an example of a mapping contracting perimeter of quadrilaterals with exactly three fixed points.

\begin{example}\label{E34}
Let $(M,d)$ be a metric space where $M=\{w,x,y,z\}$ and let $d$ be the discrete metric on $M$. Let $T:M \to M$ be defined as $Tw=x, Tx=x, Ty=y, Tz=z$. Then, $T$ is a perimetric contraction on quadrilaterals in $M$. Note that $T$ does not contain periodic points of prime period $2$ and $3$. Thus, Theorem~\ref{T33} guarantees that $T$ has a fixed point. Clearly  $Fix(T)=\{x,y,z\}$.
\end{example}
Next, we provide examples to show that neither of the conditions that $T$ has no periodic points of prime period $2$ and no periodic points of prime period $3$ can be dropped for the existence of fixed points.

\begin{example}\label{E35}
 Let $M=\{a,b,c,d\}$ be a metric space endowed with the discrete metric $d$. Let $T:M \to M$ be defined by $Ta=c, Tb=c, Tc=b, Td=b$. Then $T$ is a perimetric contraction on quadrilaterals in $M$. But, since $b$ and $c$ are periodic points of prime period $2$, therefore by Theorem~\ref{T33}, $T$ has no fixed point in $M$.
\end{example}

\begin{example}\label{E36}
Let $(M,d)$ be a metric space with $M =\{p,q,r,s\}$ and $d$ be the discrete metric on $M$. Let $T:M \to M$ be defined by $Tp=r, Tq=r, Tr=s, Ts=q$. Then $T$ is a perimetric contraction on quadrilaterals in $M$. But, since $q,r,s$ are periodic points of prime period $3$, it follows from Theorem~\ref{T33}, $T$ has no fixed point in $M$.
\end{example}

From Example \ref{E34}, we observe that a mapping contracting the perimeter of quadrilaterals may have multiple fixed points. To guarantee the existence of a unique fixed point for such a mapping, we consider an infinite complete metric space, which leads to our next result.
\begin{proposition}\label{P37}
Let $(M,d)$ be a complete metric space and let $T: M \to M$ be a perimetric contraction on quadrilaterals in $M$. If $M$ contains infinitely many points such that the iterative sequence $m_0, m_1=Tm_0, m_2=Tm_1,\dots $, converges to a point $\xi \in X$ with $\xi \neq y_i$; for all $i \in \mathbb{N} \cup \{0\}$, then $\xi$ is the unique fixed point of $T$.
\end{proposition}

\begin{proof}
That $\xi$ is a fixed point of $T$ follows from Theorem~\ref{T33}. Let $\eta$ be another fixed point of $T$. Then $\eta \neq m_i$, for all $i \in \mathbb{N} \cup \{0\}$, otherwise we have $\xi = \eta$. Therefore $\xi, \eta$, and $ m_i$ are all distinct, for all $i \in \mathbb{N} \cup \{0\}$.

Let, for all $i\in \mathbb{N}\cup \{0\}$,
\begin{align*}
K_i&= \dfrac{d(T\xi, T\eta)+ d(T\eta, Tm_{i-1})+ d(Tm_{i-1},Tm_i)+ d(Tm_i,T\xi)}{d(\xi, \eta)+ d(\eta, m_{i-1})+ d(m_{i-1},m_i)+ d(m_i,\xi)}\\
&= \dfrac{d(\xi, \eta)+ d(\eta, m_i)+ d(m_i,m_{i+1})+ d(m_{i+1},\xi)}{d(\xi, \eta)+ d(\eta, m_{i-1})+ d(m_{i-1},m_i)+ d(m_i,\xi)}.
\end{align*}
Then $K_i\le \alpha$ for all $i\in \mathbb{N}\cup \{0\}$. Now, letting $i\to \infty$, we get $K_i \to 1$ - which is a contradiction to (\ref{pcq}). 

Therefore, $T$ has a unique fixed point.
\end{proof}
We provide an alternative proof of the Banach Contraction Principle using  Theorem~\ref{T33}.
\begin{corollary}\label{C38}
(Banach Contraction Principle) Let $(M,d)$ be a complete metric space and let $T: M \to M$ be a contraction mapping, then $T$ has a unique fixed point in $M$.
\end{corollary}

\begin{proof}
If $|M|=1$ or $|M|=2$, then there is nothing to prove.

Now, for $|M|=3$, if there does not exist $m \in M$ such that $Tm=m$, then there exists $m \in M$ such that $T^2m=m$ or $T^3m=m$.

If there exists $m \in M$ such that $T^2m=m$, then $d(m, Tm)= d(Tm,m)= d(Tm, T^2m)$, a contradiction to the contraction condition. 

Again, if there exists $m \in M$ such that $T^3m=m$, then 
\begin{align*}
d(m,Tm)=d(T^3m,T^4m) &\le \alpha^3 d(m,Tm)
\end{align*}
which contradicts the contraction condition. Thus, there must be an $m \in M$ such that $Tm=m$.

Thus, if there exists $m \in M$ such that $T^2m=m$ or $T^3m=m$, then a contradiction occurs. So, there does not exist $m\in M$ such that $T^2m=m$ or $T^3m=m$.

Now, let $|M| \ge 4$. Since, there does not exist $m\in M$ such that $T^2m=m$ or $T^3m=m$ then, $T$ has no periodic points of prime period $2$ and $3$. 

Now, for all distinct points $p,q,r,s \in M$, we have
\begin{align*}
&d(Tp,Tq)+ d(Tq,Tr)+ d(Tr,Ts)+ d(Ts,Tp)\\
&\le \alpha (d(p,q)+ d(q,r)+ d(r,s)+ d(s,p)).
\end{align*}
This shows that $T$ is a perimetric contraction on quadrilaterals in $M$. Then by Theorem~\ref{T33}, it follows that $T$ admits at most three fixed points in $M$. Using the contraction condition, it can be shown that the fixed point is unique.
\end{proof}

Below, we provide a few examples to show the existence of a perimetric contraction on quadrilaterals but not a mapping contracting perimeter of triangles considering finite, countably infinite, and uncountably infinite metric spaces.  

\begin{example}\label{E39} 
Let $(M,d)$ be a metric space with $M= \left\{ 0, \frac{1}{3}, \frac{2}{3}, 1 \right\}$ where $d$ is the Euclidean metric.

Now, define the mapping $T$ as follows:
\begin{align*}
T(m)=
\begin{cases}
0, &\text{if } m=\{0, \frac{1}{3} \},\\
\frac{1}{3}, &\text{if } m= \frac{2}{3},\\
\frac{2}{3}, &\text{if } m=1.
\end{cases}
\end{align*}

Taking $a=\frac{1}{3}, b=\frac{2}{3}$, we can show that $T$ is not a contraction.
\vspace{0.3cm}

Again, for $a=\frac{1}{3}, b=\frac{2}{3}, c=1$, it contradicts (\ref{cpt})
and therefore, $T$ fails to be a mapping contracting perimeters of triangles.
\vspace{0.3cm}

Now, for any four distinct points of $M$, the the condition (\ref{pcq}) holds with  $\alpha \in \left[\frac{2}{3},1\right)$. As a result, $T$ is a perimetric contraction on quadrilaterals in $M$. Note that $T$ has no periodic points of prime order $2$ and $3$. Therefore by Theorem~\ref{T33}, $T$ has a fixed point viz., $0$.
\end{example}

In the next example, we consider a countably infinite metric space.

\begin{example}\label{E310}
Let $M=\{x^*,x_0,x_1, \dots \}$ with cardinality $\aleph_0$ and let $c \in \mathbb{R^+}$. The metric $d$ is defined on $X$ as follows:

\unitlength 1mm 
\linethickness{0.4pt}
\ifx\plotpoint\undefined\newsavebox{\plotpoint}\fi 
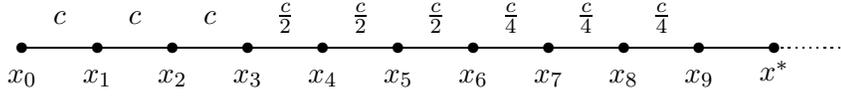
\begin{figure}[ht]
\begin{picture}(176.531,189)(0,0)
\put(20,181){\line(1,0){100}}
\multiput(120,181)(.95,0){7}{{\rule{.4pt}{.4pt}}}
\multiput(120,181)(.95,.0){11}{{\rule{.4pt}{.4pt}}}
\put(20,181){\circle*{1.5}}
\put(30,181){\circle*{1.5}}
\put(40,181){\circle*{1.5}}
\put(50,181){\circle*{1.5}}
\put(60,181){\circle*{1.5}}
\put(70,181){\circle*{1.5}}
\put(80,181){\circle*{1.5}}
\put(90,181){\circle*{1.5}}
\put(100,181){\circle*{1.5}}
\put(110,181){\circle*{1.5}}
\put(120,181){\circle*{1.5}}
\put(20,177){\makebox(0,0)[cc]{$x_0$}}
\put(30,177){\makebox(0,0)[cc]{$x_1$}}
\put(40,177){\makebox(0,0)[cc]{$x_2$}}
\put(50,177){\makebox(0,0)[cc]{$x_3$}}
\put(60,177){\makebox(0,0)[cc]{$x_4$}}
\put(70,177){\makebox(0,0)[cc]{$x_5$}}
\put(80,177){\makebox(0,0)[cc]{$x_6$}}
\put(90,177){\makebox(0,0)[cc]{$x_7$}}
\put(100,177){\makebox(0,0)[cc]{$x_8$}}
\put(110,177){\makebox(0,0)[cc]{$x_9$}}
\put(120,178){\makebox(0,0)[cc]{$x^*$}}
\put(25,185){\makebox(0,0)[cc]{$c$}}
\put(35,185){\makebox(0,0)[cc]{$c$}}
\put(45,185){\makebox(0,0)[cc]{$c$}}
\put(55,185){\makebox(0,0)[cc]{$\frac{c}{2}$}}
\put(65,185){\makebox(0,0)[cc]{$\frac{c}{2}$}}
\put(75,185){\makebox(0,0)[cc]{$\frac{c}{2}$}}
\put(85,185){\makebox(0,0)[cc]{$\frac{c}{4}$}}
\put(95,185){\makebox(0,0)[cc]{$\frac{c}{4}$}}
\put(105,185){\makebox(0,0)[cc]{$\frac{c}{4}$}}
\end{picture}
\vspace{-47em}
\caption{ The points in the space $(M,d)$ that are separated by the consecutive lengths.}
\end{figure}

\begin{align*}
d(x,y)=
\begin{cases}
\frac{c}{2^{[{\frac{n}{3}}]}}, &\text{if } x=x_n, y=x_{n+1},n=1,2,\dots,\\
\sum_{a=n}^{m-1} \frac{c}{2^{[{\frac{a}{3}}]}}, & \text{if } x=x_n, y=x_m, n+1<m,\\
6c- \sum_{a=0}^{n-1} \frac{c}{2^{[{\frac{a}{3}}]}}, & \text{if } x=x_n, y=x^*,\\
0, & \text{if } x=y,
\end{cases}
\end{align*}
where $[\cdot]$  is the box function.
Then, $(M,d)$ is a complete metric space.

Define a mapping $T:M \to M$ as $Tx_n=x_{n+1}$ for all $n=\mathbb{N} \cup \{0\}$ and $Tx^*=x^*$.

Since $d(Tx_{3n}, Tx_{3n+1})= d(x_{3n},x_{3n+1})$ for all $n=\mathbb{N} \cup \{0\}$, then $T$ is not a contraction.

Also,
\begin{align*}
&d(Tx_{3n}, Tx_{3n+1})+ d(Tx_{3n+1}, Tx_{3n+2})+ d(Tx_{3n+2}, Tx_{3n})\\
=&d(x_{3n}, x_{3n+1})+ d(x_{3n+1}, x_{3n+2})+ d(x_{3n+2},x_{3n}) 
\end{align*}
for all $n=\mathbb{N} \cup \{0\}$. Therefore, $T$ fails to be a mapping contracting perimeters of triangles. 

We now show that $T$ is a perimetric contraction on quadrilaterals in $M$. 

Consider the points $x_k, x_l, x_m, x^* \in X$ with $0 \le k < l < m$. Then, we have,
\begin{align*}
d(x_k, x_l)+ d(x_l, x_m)+ d(x_m, x^*)+ d(x^*, x_k)= 2d(x_k,x^*)
= 12c- 2 \sum_{a=0}^{k-1} \frac{c}{2^{[{\frac{a}{3}}]}} 
\end{align*}
and 
\begin{align*}
&d(Tx_k, Tx_l)+ d(Tx_l, Tx_m)+ d(Tx_m, Tx^*)+ d(Tx^*, Tx_k)\\
=&2d(Tx_k,Tx^*)= 2d(x_{k+1},x^*) =12c- 2 \sum_{a=0}^{k} \frac{c}{2^{[{\frac{a}{3}}]}}.
\end{align*}
Now, we have,
\begin{align}
d(x_0, x_n)= 
\begin{cases}
6c(1-({\frac{1}{2}})^p), &\text{if } n=3p,\\
6c(1-({\frac{1}{2}})^p)- \frac{c}{2^{p-1}}, & \text{if } n=3p-1,\\
6c(1-({\frac{1}{2}})^p)- \frac{c}{2^{p-2}}, & \text{if } n=3p-2.\label{3.11}
\end{cases}
\end{align}

Consider the ratio,
\begin{align*}
R_k=&\frac{d(Tx_k, Tx_l)+ d(Tx_l, Tx_m)+ d(Tx_m, Tx^*)+ d(Tx^*, Tx_k)}{d(x_k, x_l)+ d(x_l, x_m)+ d(x_m, x^*)+ d(x^*, x_k)}\\
=&\frac{12c- 2 \sum_{a=0}^{k} \frac{c}{2^{[{\frac{a}{3}}]}}}{12c- 2 \sum_{a=0}^{k-1} \frac{c}{2^{[{\frac{a}{3}}]}}}\\
=&1- \frac{\frac{c}{2^{\left[{\frac{k}{3}}\right]}}}{6c- \sum_{a=0}^{k-1} \frac{c}{2^{[{\frac{a}{3}}]}}}\\
=&\begin{cases}
\vspace{0.2cm}
1- \frac{\frac{c}{2^p}}{6c- 6c(1-({\frac{1}{2}})^p)}, &\text{if } k=3p,\\
\vspace{0.2cm}
1- \frac{\frac{c}{2^{p-1}}}{6c- 6c(1-({\frac{1}{2}})^p)+ \frac{c}{2^{p-1}}}, & \text{if } k=3p-1,\\
1- \frac{\frac{c}{2^{p-1}}}{6c- 6c(1-({\frac{1}{2}})^p)+ \frac{c}{2^{p-2}}}, & \text{if } k=3p-2,
\end{cases}\\
=&\begin{cases}
\frac{5}{6}, &\text{if } k=3p,\\
\frac{3}{4}, &\text{if } k=3p-1,\\
\frac{4}{5}, &\text{if } k=3p-2.
\end{cases}
\end{align*}

Now, let us consider the points $x_k, x_l, x_m, x_n \in X$ with $0 \le k < l < m < n$. Then, we have,
\begin{align*}
&d(x_k, x_l)+ d(x_l, x_m)+ d(x_m, x_n)+ d(x_n, x_k)= 2d(x_k,x_n)= 2 \sum_{a=k}^{n-1} \frac{c}{2^{[{\frac{a}{3}}]}}
\end{align*}
and
\begin{align*}
&d(Tx_k, Tx_l)+ d(Tx_l, Tx_m)+ d(Tx_m, Tx_n)+ d(Tx_n, Tx_k)\\
=&2d(Tx_k,Tx_n)\\
=&2d(x_{k+1},x_{n+1})\\
=&2d(x_k,x_n) - 2[d(x_k, x_{k+1})- d(x_n, x_{n+1})]\\
=&2 \sum_{a=k}^{n-1} \frac{c}{2^{[{\frac{a}{3}}]}} - 2\left(\frac{c}{2^{[{\frac{k}{3}]}}} - \frac{c}{2^{[{\frac{n}{3}]}}}\right).
\end{align*}
Consider the ratio,
\begin{align*}
R_{k,n}&= \frac{d(Tx_k, Tx_l)+ d(Tx_l, Tx_m)+ d(Tx_m, Tx_n)+ d(Tx_n, Tx_k)}{d(x_k, x_l)+ d(x_l, x_m)+ d(x_m, x_n)+ d(x_n, x_k)}\\
&= \frac{2 \sum_{a=k}^{n-1} \frac{c}{2^{[{\frac{a}{3}}]}} - 2\left(\frac{c}{2^{[{\frac{k}{3}]}}} - \frac{c}{2^{[{\frac{n}{3}]}}}\right)}{2 \sum_{a=k}^{n-1} \frac{c}{2^{[{\frac{a}{3}}]}}}\\
&=1- \frac{\frac{c}{2^{[{\frac{k}{3}]}}} - \frac{c}{2^{[{\frac{n}{3}]}}}}{\sum_{a=k}^{n-1} \frac{c}{2^{[{\frac{a}{3}}]}}}
\end{align*}
It is to be noted that $n\ge k+3$. Therefore
\begin{align}
\left[\frac{n}{3}\right] \ge \left[\frac{k}{3}\right]+1
\implies 2^{\left[\frac{n}{3}\right]} \ge 2.2^{\left[\frac{k}{3}\right]}
\implies \frac{1}{2^{\left[\frac{n}{3}\right]}} \le \frac{1}{2.2^{\left[\frac{k}{3}\right]}}
\implies \frac{c}{2^{\left[\frac{n}{3}\right]}} \le \frac{c}{2.2^{\left[\frac{k}{3}\right]}}\label{3.12}
\end{align}
Now from (\ref{3.11}), we can write,
\begin{align}
d(x_n, x^*)= 
\begin{cases}
{\frac{6c}{2^p}}, &\text{if } n=3p,\\
{\frac{6c}{2^p}}+ \frac{c}{2^{p-1}}, & \text{if } n=3p-1,\\
{\frac{6c}{2^p}}+ \frac{c}{2^{p-2}}, & \text{if } n=3p-2.\label{3.13}
\end{cases}
\end{align}

Therefore from (\ref{3.13}), we get,
\begin{align}
&d(x_n,x^*) \le 6 d(x_n, x_{n+1})\notag\\
\implies & d(x_k,x^*) \le 6 d(x_k, x_{k+1})\notag\\
\implies & d(x_k, x_n) \le d(x_k,x^*) \le 6 d(x_k, x_{k+1})\notag\\
\implies &\sum_{a=k}^{n-1} \frac{c}{2^{[{\frac{a}{3}}]}} \le 6 \frac{c}{2^{[{\frac{k}{3}]}}}.\label{3.14}
\end{align}

Consequently, from (\ref{3.12}) and (\ref{3.14}) we have
\begin{align*}
R_{k,n} \le 1- \frac{\frac{c}{2^{[\frac{k}{3}]}} - \frac{1}{2} \frac{c}{2^{[\frac{k}{3}]}}}{6 \frac{c}{2^{[\frac{k}{3}]}}} = \frac{11}{12}.
\end{align*}
Thus, the inequality (\ref{pcq}) holds for any four distinct points from $X$ with $\alpha= \frac{11}{12}= \max\{ \frac{5}{6}, \frac{3}{4}, \frac{4}{5}, \frac{11}{12} \}$. 

Therefore, $T$ is a perimetric contraction on quadrilaterals in $M$. Also, $T$ does not contain any periodic points of prime period $2$ and $3$. Hence, by Theorem~\ref{T33}, $T$ has a fixed point in $M$. Note that $x^* \in Fix(T)$.
\end{example}

In the next example, we consider an uncountably infinite metric space.

\begin{example}\label{E311}
Let $M=\{ -1, -\frac{2}{3}, -\frac{1}{3} \} \cup [0,1] \subset \mathbb{R}$ is a metric space equipped with the Euclidean metric $d$ and let $T:M \to M$ be a mapping defined as follows:
\begin{align*}
T(m)=
\begin{cases}
\frac{m}{2}, &\text{if } m \in [0,1],\\
0, &\text{if } m=- \frac{1}{3},\\
- \frac{1}{3}, &\text{if } m=- \frac{2}{3},\\
- \frac{2}{3}, &\text{if } m=-1.
\end{cases}
\end{align*}

Now, for $p=-\frac{1}{3}, q=-\frac{2}{3}$, we see that $T$ is not a contraction.

Again, for $p=-\frac{1}{3}, q=-\frac{2}{3}, r=-1$, we can show that $T$ is not a mapping contracting perimeters of triangles.

For any four distinct points of $M$, the the condition (\ref{pcq}) holds for $\alpha \in [\frac{2}{3}, 1)$. Thus, $T$ is a perimetric contraction on quadrilaterals in $M$. Also, $T$ does not contain any periodic points of prime periods 2 and 3. Therefore, by Theorem~\ref{T33}, $Fix(T)=\{0\}$.
\end{example}

\section{Kannan type perimetric contraction quadrilaterals} 
We begin the section by introducing the four-point analogue of the Kannan type contraction in the following way:

\begin{definition}\label{D41}
Let $(M,d)$ be a metric space with at least fore points. Then a mapping $T:M \to M$ is called a Kannan type perimetric contraction quadrilaterals on $M$ if there exists $\delta \in [0,\frac{1}{2})$ such that the following inequality holds for all distinct points $p,q,r,s \in M$
\begin{align}
&d(Tp,Tq)+ d(Tq,Tr)+ d(Tr,Ts)+ d(Ts,Tp)\notag\\
&\le \delta (d(p,Tp)+ d(q,Tq)+ d(r,Tr)+ d(s,Ts)) \label{qkc}
\end{align}
\end{definition}

The following result is a direct consequence of the definitions.

\begin{theorem}\label{T42}
Let $(M, d)$ be a metric space with at least four points and let $T: M \to M$ be a Kannan type mapping with $\delta \in [0,\frac{1}{4})$. Then $T$ is a Kannan type perimetric contraction on quadrilaterals.
\end{theorem}



In the following result, we show that a sub-collection of perimetric contraction on quadrilaterals is contained in the collection of Kannan type perimetric contraction on quadrilaterals. 
\begin{theorem}\label{T43}
Let $(M,d)$ be a metric space with at least four points and let $T: M \to M$ be a perimetric contraction on quadrilaterals with $\delta \in [0,\frac{1}{5})$. Then $T$ is a Kannan type perimetric contraction on quadrilaterals.
\end{theorem}

\begin{proof}
Let $T$ be a perimetric contraction on quadrilaterals with $\delta \in [0,\frac{1}{5})$ and let $w,x,y,z \in M$ be distinct. Then using (\ref{qkc}), we have
\begin{align*}
&d(Tp,Tq)+ d(Tq,Tr)+ d(Tr,Ts)+ d(Ts,Tp)\\
\le &\delta (d(p,q)+ d(q,r)+ d(r,s)+ d(s,p))\\
\le &\delta (d(p,Tp)+ d(Tp,Tq)+ d(Tq,q)+ d(q,Tq)+ d(Tq,Tr)+ d(Tr,r)\\
& + d(r,Tr)+ d(Tr,Ts)+ d(Ts,s)+ d(s,Ts)+ d(Ts,Tp)+ d(Tp,p)).
\end{align*}
This implies
\begin{align*}
&d(Tp,Tq)+ d(Tq,Tr)+ d(Tr,Ts)+ d(Ts,Tp)\\
&\le \frac{2\delta}{1-\delta} (d(p,Tp)+ d(q,Tq)+ d(r,Tr)+ d(s,Ts)).
\end{align*}
Thus, $T$ is a Kannan type perimetric contraction on quadrilaterals as $\frac{2\delta}{1-\delta} \in [0,\frac{1}{2})$.
\end{proof}
Now, we obtain a sufficient condition for the existence of fixed point(s) of Kannan type perimetric contraction quadrilaterals.
\begin{theorem}\label{T44}
Let $(M,d)$ be a complete metric space with at least four points, and let $T: M \to M$ be a Kannan type perimetric contraction on quadrilaterals. Then $T$ admits a fixed point if $T$ does not possess periodic points of prime period $2$ and $3$. The number of fixed points is at most three.
\end{theorem}

\begin{proof}
Let $T: M \to M$ be a Kannan type perimetric contraction on quadrilaterals on $M$ and let $T$ have no periodic points of prime period $2$ and $3$. 

Let $a_0 \in X$ be chosen arbitrarily. Define $Ta_0=a_1, Ta_1=a_2, \cdots, Ta_n=a_{n+1}, \cdots$. If $a_i$ is a fixed point of $T$ for any $i$, then there is nothing to show.

Since $a_i$ are not fixed points of $T$ and $T$ have no periodic points of prime period $2$ and $3$, following the same way as in Theorem~\ref{T33}, we can show that any four consecutive elements of the sequence $\{a_n\}$ are distinct. Now, we have


\begin{align*}
&d(Ta_i,Ta_{i+1})+d(Ta_{i+1},Ta_{i+2})+d(Ta_{i+2},Ta_{i+3})+d(Ta_{i+3},Ta_i)\\
&\le \delta (d(a_i,Ta_i)+ d(a_{i+1},Ta_{i+1})+ d(a_{i+2},Ta_{i+2})+ d(a_{i+3},Ta_{i+3}))\\
\implies &d(a_{i+1},a_{i+2})+d(a_{i+2},a_{i+3})+d(a_{i+3},a_{i+4})+d(a_{i+4},a_{i+1})\\
&\le \delta (d(a_i,a_{i+1})+ d(a_{i+1},a_{i+2})+ d(a_{i+2},a_{i+3})+ d(a_{i+3},a_{i+4}))\\
\implies &(1-\delta)d(a_{i+3},a_{i+4})\le \delta (d(a_i,a_{i+1})+ d(a_{i+1},a_{i+2})+ d(a_{i+2},a_{i+3}))\\ 
&\hspace{3.5cm}- (d(a_{i+1},a_{i+2})+d(a_{i+2},a_{i+3})+d(a_{i+4},a_{i+1})\\
\implies &(1-\delta)d(a_{i+3},a_{i+4})\le \delta (d(a_i,a_{i+1})+ d(a_{i+1},a_{i+2})+ d(a_{i+2},a_{i+3}))\\ 
&\hspace{3.5cm}- d(a_{i+3},a_{i+4})\\
\implies &(2-\delta)d(a_{i+3},a_{i+4})\le \delta (d(a_i,a_{i+1})+ d(a_{i+1},a_{i+2})+ d(a_{i+2},a_{i+3}))\\ 
\implies &d(a_{i+3},a_{i+4}) \le \frac{\delta}{2-\delta} (d(a_i,a_{i+1})+ d(a_{i+1},a_{i+2})+ d(a_{i+2},a_{i+3}))\\
\implies &d(a_{i+3},a_{i+4}) \le \frac{3\delta}{2-\delta} \max\{d(a_i,a_{i+1}),d(a_{i+1},a_{i+2}),d(a_{i+2},a_{i+3})\}.
\end{align*}

Let $\lambda= \frac{3\delta}{2-\delta}$. Then $\lambda \in [0,1)$ as $\delta \in [0,\frac{1}{2})$. Also, let $k_i=d(a_i,a_{i+1}), i \in \mathbb{N} \cup \{0\}$ and $k=\max\{k_1,k_2,k_3\}$. Then, for any $i \in \mathbb{N} \cup \{0\}$, we have
\begin{align*}
    &d(a_{i+3},a_{i+4}) \le \lambda \max\{d(a_i,a_{i+1}),d(a_{i+1},a_{i+2}),d(a_{i+2},a_{i+3})\}\\
    \implies & k_{i+3} \le \lambda \max\{k_i,k_{i+1},k_{i+2}\}.
\end{align*}

Consequently, we have
\begin{align*}
&k_1 \le k, k_2 \le k, k_3 \le k,
k_4 \le \lambda k, k_5 \le \lambda k, k_6 \le \lambda k, k_7 \le {\lambda}^2 k, \cdots.
\end{align*}

Since $\lambda <1$, so we have
\begin{align*}
&k_1 \le k, k_2 \le k, k_3 \le k,
k_4 \le {\lambda}^{\frac{1}{3}} k, k_5 \le {\lambda}^{\frac{2}{3}} k, k_6 \le \lambda k, k_7 \le {\lambda}^{\frac{4}{3}} k, \cdots.\\
\implies & k_n \le k {\lambda}^{\frac{n}{3}-1}, \text{ for all } n \in \mathbb{N} \text{ with } n \ge 4.
\end{align*}

Now, for all $n \in \mathbb{N} \cup \{0\}$ and for any $p=1,2,\cdots$, we have
\begin{align*}
d(a_n,a_{n+p}) &\le d(a_n,a_{n+1})+ d(a_{n+1},a_{n+2})+ \dots + d(a_{n+p-1},a_{n+p})\\
&\le k_n+ k_{n+1}+ \cdots + k_{n+p-1}\\
&\le k(\lambda^{\frac{n}{3}-1} + \lambda^{\frac{n+1}{3}-1} + \cdots + \lambda^{\frac{n+p-1}{3}-1}) \\
&\le k \lambda^{\frac{n}{3}-1} \dfrac{1- \lambda^{\frac{p}{3}}}{1- \lambda^{\frac{1}{3}}}.
\end{align*}

Therefore, $d(a_n,a_{n+p}) \to 0$ as $n \to \infty$ and for any $p=1,2, \cdots$. Thus, $\{a_n\}$ is a Cauchy sequence in $M$ and consequently, by completeness of $M$, there exists an $a^* \in M$ such that $a_n \to a^*$ as $n\to \infty$. 

Now,
\begin{align*}
&d(a^*,Ta^*)\\
&\le d(a^*,a_n)+ d(a_n,Ta^*)\\
&\le d(a^*,a_n)+ d(Ta_{n-1},Ta^*)\\
&\le d(a^*,a_n)+ d(Ta^*,Ta_{n-1})+ d(Ta_{n-1},Ta_{n-2})+ d(Ta_{n-2},Ta_n)+ d(Ta_n,Ta^*)\\
&\le d(a^*,a_n)+ \delta (d(a^*,Ta^*)+d(a_{n-1},Ta_{n-1})+ d(a_{n-2},Ta_{n-2})+ d(a_n,Ta_n).
\end{align*}
This implies that
\begin{align*}
(1-\delta) d(a^*,Ta^*) &\le d(a^*,a_n)+ \delta (d(a_{n-1},a_n)+ d(a_{n-2},a_{n-1})+ d(a_n,a_{n+1}).
\end{align*}

Taking $n \to \infty$, we get $Ta^*=a^*$ \textit{i.e.} $a^* \in Fix(T)$.

If possible, suppose that $T$ has four distinct fixed points say $p,q,r,s \in M$. 
Then from (\ref{qkc}), we have 
\begin{align*}
d(p,q)+ d(q,r)+ d(r,s)+ d(s,p) \le 0,
\end{align*}
which contradicts the fact that $p,q,r,s$ all are distinct. 

Thus, $T$ can have at most three fixed points. 
\end{proof}

Below, we provide an example to show the existence of a Kannan type perimetric contraction on quadrilaterals with exactly three fixed points.  

\begin{example}\label{E45} 
Let $M=\{a,b,c,d \}$ and let the metric $d$ be defined on $M$ by,

$d(a,b)= d(b,c)= d(a,c)= 1$ and $d(a,d)= d(b,d)= d(c,d)= 8$.

We now define the mapping $T:M \to M$ by $Ta=a, Tb=b, Tc=c, Td=a$.

Then $(M,d)$ is a complete metric space and $T$ is a Kannan type perimetric contraction on quadrilaterals on $M$. Also, $T$ does not contain any periodic points of prime period $2$ and $3$. Therefore by Theorem~\ref{T44}, $T$ has three fixed points and $Fix(T)= \{a,b,c\}$.
\end{example}
While the converse of Theorem ~\ref{T44} does not necessarily hold in general, a straightforward proof of a partial converse exists, which we omit here. 
\begin{theorem}\label{T46}
Let $(M,d)$ be a complete metric space with at least four points, and let $T: M \to M$ be a Kannan type perimetric contraction quadrilaterals. If $T$ possesses a fixed point in $M$, then $T$ does not possess periodic points of prime period $3$. 
\end{theorem}



It is obvious from Example~\ref{E45} that fixed points of Kannan type perimetric contraction on quadrilaterals may not be unique. To confirm the existence of a unique fixed point of such a mapping, we must take an infinite complete metric space which is our next result:
\begin{proposition}\label{P47}
Let $(M,d)$ be a complete metric space and $T: M \to M$ be a Kannan type perimetric contraction quadrilaterals. If $M$ contains infinitely many points such that the iterative sequence $m_0, m_1=Tm_0, m_2=Tm_1,\dots $, converges to a point $\xi$ with $\xi \neq m_n$; for all $n \in \mathbb{N} \cup \{0\}$. Then $\xi$ is the unique fixed point of $T$.
\end{proposition}

\begin{proof}
It follows from Theorem~\ref{T44} that $T$ has a fixed point $\xi$. Let $\eta$ be another fixed point of $T$. Then $\eta \neq m_n$, for all $n \in \mathbb{N} \cup \{0\}$, otherwise we have $\xi = \eta$ . Therefore $\xi, \eta$, and $ m_n$ are all distinct, for all $n \in \mathbb{N} \cup \{0\}$.

Let, for all $n \in \mathbb{N}\cup \{0\}$,
\begin{align*}
&d(T\xi, Tm_n)+ d(Tm_n, T\eta)+ d(T\eta,Tm_{n+1})+ d(Tm_{n+1},T\xi)\\
&\le \delta_n (d(\xi, T\xi)+ d(m_n, Tm_{n})+ d(\eta, T\eta)+ d(m_{n+1}, Tm_{n+1})).
\end{align*}
This implies,
\begin{align*}
&d(\xi, m_{n+1})+ d(m_{n+1}, \eta)+ d(\eta,m_{n+2})+ d(m_{n+2},\xi)\\
&\le \delta_n (d(m_n, m_{n+1})+  d(m_{n+1}, m_{n+2})).
\end{align*}
Now, letting $n \to \infty$, we get 
\begin{align*}
2 d(\xi, \eta) \le 0,
\end{align*}
which is a contradiction to the fact that $\xi \neq \eta$. 

Therefore, $T$ has a unique fixed point in $M$.
\end{proof}
Below, we mention an example of \cite{PB24} to distinguish the classes of Kannan type mappings, generalized Kannan type mappings, and Kannan type perimetric contraction on quadrilaterals.
\begin{example}\label{E48}\cite{PB24}
Let $M=[0,1]$ and let $d$ be the Euclidean metric on $M$. Consider a mapping $T:M \to M$ defined by $T(x)=\frac{x}{k}$ for some $k>1$.

In \cite{PB24}, it is shown that $T$ is a Kannan type mapping for $k>3$ and $T$ is a generalized Kannan type mapping for $k>4$.

Without loss of generality, let us suppose that $a,b,c,d \in [0,1]$ with $a>b>c>d$. Therefore, by (\ref{qkc}) we have 
\begin{align*}
&\frac{1}{k}(a-b+b-c+c-d+a-d) \le \lambda (1-\frac{1}{k}) (a+b+c+d)\\
\implies &2(a-d) \le \lambda (k-1) (a+b+c+d)\\
\implies &a-d \le \frac{\lambda}{2} (k-1) (a+b+c+d).
\end{align*}
The above inequality holds for all $a,b,c,d \in [0, 1]$ with $a>b>c>d$ if and only if $\frac{\lambda}{2} (k-1)\ge 1$.
Also, $T$ is a Kannan type perimetric contraction on quadrilaterals if and only if $\lambda \in [0,\frac{1}{2})$.
Consequently,
\begin{align*}
\frac{1}{2} > \lambda \ge \frac{2}{k-1}.
\end{align*}
As a result, $T$ is a Kannan type perimetric contraction on quadrilaterals for $k>5$.
\end{example}
Now, we give an example in support of Theorem~\ref{T44}.
\begin{example}\label{E49}
Let $M=\{a,b,c,d \}$ and the metric $d$ be defined by 
\begin{align*}
d(x,y)=
\begin{cases}
0, &\text{if } x=y,\\
1, &\text{if } (x,y)=(a,b) \text{ or } (x,y)=(b,a),\\
3, &\text{otherwise}.
\end{cases}
\end{align*}
Then, $(M,d)$ is a complete metric space.\\
Consider $T:M \to M$ be defined by $Ta=a, Tb=b, Tc=b, Td=a$.\\
Taking $x=a$ and $y=b$, it can be shown that $T$ is not a Kannan type mapping. 

If $T$ is a generalized Kannan type mapping, then we have
\begin{align*}
&d(Ta,Tb)+ d(Tb,Tc)+ d(Tc,Ta) \le \delta (d(a,Ta)+ d(b,Tb)+ d(c,Tc))\\
\implies &d(a,b)+ d(b,b)+ d(b,a) \le \delta (d(a,a)+ d(b,b)+ d(c,b))\\
\implies &2 d(a,b) \le \delta  d(b,c)\\
\implies & 2 \le 3\delta,
\end{align*}
which is a contradiction to (\ref{gkc}).
Thus, $T$ is not a generalized Kannan type contraction.

Since $T$ satisfy $(\ref{qkc})$ with $\delta \in [\frac{1}{3}, \frac{1}{2})$, it follows that $T$ is a Kannan type perimetric contraction on quadrilaterals. 
Note that, $T$ does not contain any periodic point of prime period $2$ and $3$. Therefore by Theorem~\ref{T44}, $T$ has fixed points and $Fix(T)=\{a,b\}$.
\end{example}
Below, we provide an example to show that containing periodic points of prime period $2$ restricts from attaining a fixed point.
\begin{example}\label{E410}
Let $(M,d)$ be a metric space with $M=\{p,q,r,s\}$ and the metric $d$ be defined as 
\begin{align*}
d(x,y)=
\begin{cases}
0, &\text{if } x=y,\\
1, &\text{if } (x,y)=(p,r) \text{ or } (x,y)=(r,p),\\
3, &\text{otherwise}.
\end{cases}
\end{align*}
Let us define the mapping $T$ as $Tp=r, Tq=s, Tr=p, Ts=r$.\\
Then, $T$ satisfy $(\ref{qkc})$ with $\delta \in [\frac{1}{3}, \frac{1}{2})$. Thus $T$ is a Kannan type perimetric contraction on quadrilaterals. Also, $p$ and $r$ are two periodic points of $T$ of prime period $2$. Note that, $T$ has no fixed point in $M$.
\end{example}

\section{Chatterjea type perimetric contraction quadrilaterals}
We begin the section with the introduction of the four-point analogue of Chatterjea type contraction as follows:

\begin{definition}\label{D51}
Let $(M,d)$ be a metric space with at least four points. Then a mapping $T: M \to M$ is said to be a Chatterjea type perimetric contraction on quadrilaterals if there exists $\lambda \in [0,\frac{1}{7})$ such that 
\begin{align}
&d(Tp,Tq)+ d(Tq,Tr)+ d(Tr,Ts)+ d(Ts,Tp)\notag\\
&\le \lambda (d(p,Tq)+ d(p,Ts)+ d(q,Tp)+ d(q,Tr)\notag\\
&+ d(r,Tq)+ d(r,Ts)+ d(s,Tr)+ d(s,Tp)) \label{qcc}
\end{align}
for all distinct points $p,q,r,s \in M$.
\end{definition}
The following result is a direct consequence of the definitions.
\begin{theorem}\label{T52}
Let $(M, d)$ be a metric space with at least four points and let $T: M \to M$ be a Chatterjea type mapping with $\lambda \in [0,\frac{1}{7})$. Then $T$ is a Chatterjea type perimetric contraction on quadrilaterals.
\end{theorem}



Below, we establish a result to show that a sub-collection of mapping contracting perimeter of quadrilaterals is contained in the collection of Chatterjea type perimetric contraction on quadrilaterals.
\begin{theorem}\label{T53}
Consider a complete metric space $(M, d)$ with at least four points and let $T: M \to M$ be a mapping contracting perimeter of quadrilaterals with $\lambda \in [0,\frac{1}{8})$. Then $T$ is a Chatterjea type perimetric contraction on quadrilaterals.
\end{theorem}

\begin{proof}
Let $p,q,r,s \in M$ be distinct. Then, from (\ref{pcq}) we get
\begin{align*}
&d(Tp,Tq)+ d(Tq,Tr)+ d(Tr,Ts)+ d(Ts,Tp)\\
&\le \lambda (d(p,q)+ d(q,r)+ d(r,s)+ d(s,p))\\
& \le \lambda (d(p,Tq)+ d(Tq,Tp)+ d(Tp,q)+ d(q,Tr)+ d(Tr,Tq)+ d(Tq,r)\\
&+ d(r,Ts)+ d(Ts,Tr)+ d(Tr,s)+ d(s,Tp)+ d(Tp,Ts)+ d(Ts,p)).
\end{align*}
This implies
\begin{align*}
&d(Tp,Tq)+ d(Tq,Tr)+ d(Tr,Ts)+ d(Ts,Tp)\\
&\le \frac{\lambda}{1-\lambda} (d(p,Tq)+ d(p,Ts)+ d(q,Tp)+ d(q,Tr)\\
&\hspace{1.5cm}+ d(r,Tq)+ d(r,Ts)+ d(s,Tr)+ d(s,Tp)).
\end{align*}
Thus, $T$ is a Chatterjea type perimetric contraction on quadrilaterals as $\frac{\lambda}{1-\lambda} \in [0,\frac{1}{7})$.
\end{proof}
Now, we present the following result that ensures the existence of fixed point(s) of Chatterjea type perimetric contraction on quadrilaterals.
\begin{theorem}\label{T54}
Consider a complete metric space $(M,d)$ with at least four points, and let $T: M \to M$ be a Chatterjea type perimetric contraction on quadrilaterals on $M$. Then $T$ attains a fixed point if $T$ does not possess periodic points of prime period $2$ and $3$. The number of fixed points is at most three.
\end{theorem}

\begin{proof}
Let $T:M \to M$ be a Chatterjea type perimetric contraction quadrilaterals and let $T$ have no periodic points of prime period $2$ and $3$.

Let $a_0 \in M$ be chosen arbitrarily. Define $Ta_0=a_1, Ta_1=a_2, \cdots, Ta_n=a_{n+1}, \cdots$. If $a_i$ is a fixed point of $T$ for any $i\in \mathbb{N} \cup \{0\}$, then we are done.

Since $a_i$ are not fixed points of $T$ and $T$ have no periodic points of prime period $2$ and $3$, similarly as in Theorem~\ref{T33}, we can show that any four consecutive elements of the sequence $\{a_n\}$ are distinct. Now, we have


\begin{align*}
&d(Ta_i,Ta_{i+1})+d(Ta_{i+1},Ta_{i+2})+d(Ta_{i+2},Ta_{i+3})+d(Ta_{i+3},Ta_i)\\
&\le \lambda (d(a_i,Ta_{i+1})+ (d(a_i,Ta_{i+3})+ d(a_{i+1},Ta_{i})+ d(a_{i+1},Ta_{i+2})\\
& + d(a_{i+2},Ta_{i+1})+ d(a_{i+2},Ta_{i+3})+ d(a_{i+3},Ta_{i+2})+ d(a_{i+3},Ta_{i}))\\
\implies &d(a_{i+1},a_{i+2})+d(a_{i+2},a_{i+3})+d(a_{i+3},a_{i+4})+d(a_{i+4},a_{i+1})\\
&\le \lambda (d(a_i,a_{i+2})+ (d(a_i,a_{i+4})+ d(a_{i+1},a_{i+1})+ d(a_{i+1},a_{i+3})\\
&+ d(a_{i+2},a_{i+2})+ d(a_{i+2},a_{i+4})+ d(a_{i+3},a_{i+3})+ d(a_{i+3},a_{i+1}))\\
\implies &d(a_{i+3},a_{i+4})\le \lambda (d(a_i,a_{i+2})+ (d(a_i,a_{i+4})+ d(a_{i+1},a_{i+3})\\
& \hspace{2.5cm}+ d(a_{i+2},a_{i+4})+ d(a_{i+3},a_{i+1})) - d(a_{i+3},a_{i+4})\\
\implies & 2d(a_{i+3},a_{i+4})\le \lambda (d(a_i,a_{i+1})+ d(a_{i+1},a_{i+2})+ (d(a_i,a_{i+1})+ d(a_{i+1},a_{i+2})\\
&\hspace{2.5cm}+ d(a_{i+2},a_{i+3})+ d(a_{i+3},a_{i+4}) d(a_{i+1},a_{i+2})+ d(a_{i+2},a_{i+3})\\
&\hspace{2.5cm}+ d(a_{i+2},a_{i+3})+ d(a_{i+3},a_{i+4})+ d(a_{i+1},a_{i+2})+ d(a_{i+2},a_{i+3})\\
\implies &2(1-\lambda)d(a_{i+3},a_{i+4})\le  \lambda (2d(a_i,a_{i+1})+ 4d(a_{i+1},a_{i+2})+ 4d(a_{i+2},a_{i+3})\\ 
\implies &d(a_{i+3},a_{i+4}) \le \frac{2\lambda}{1-\lambda} (d(a_i,a_{i+1})+ d(a_{i+1},a_{i+2})+ d(a_{i+2},a_{i+3})\\
\implies &d(a_{i+3},a_{i+4}) \le \frac{6\lambda}{1-\lambda} \max\{d(a_i,a_{i+1}),d(a_{i+1},a_{i+2}),d(a_{i+2},a_{i+3}\}
\end{align*}

Suppose that $\delta= \frac{6\lambda}{1-\lambda}$ then $\delta \in [0,1)$ as $\lambda \in [0,\frac{1}{7})$.

Also take, $k_i=d(a_i,a_{i+1})$ and $k=\max\{k_1,k_2,k_3\}$.

So, we can write, 
\begin{align*}
    &d(a_{i+3},a_{i+4}) \le \delta \max\{d(a_i,a_{i+1}),d(a_{i+1},a_{i+2}),d(a_{i+2},a_{i+3})\}\\
    \implies & k_{i+3} \le \delta \max\{k_i,k_{i+1},k_{i+2}\}.
\end{align*}

Hence, we get 
\begin{align*}
&k_1 \le k, k_2 \le k, k_3 \le k,
k_4 \le \delta k, k_5 \le \delta k, k_6 \le \delta k, k_7 \le {\delta}^2 k, \cdots.
\end{align*}

Since $\lambda \in [0,1)$, thus we have 
\begin{align*}
&k_1 \le k, k_2 \le k, k_3 \le k,
k_4 \le {\delta}^{\frac{1}{3}} k, k_5 \le {\delta}^{\frac{2}{3}} k, k_6 \le \delta k, k_7 \le {\delta}^{\frac{4}{3}} k, \cdots.\\
\implies & k_n \le {\delta}^{\frac{n}{3}-1} k, \text{ for all } n \in \mathbb{N} \text{ with } n \ge 4.
\end{align*}

Now, for all $n \in \mathbb{N} \cup \{0\}$ and for any $p=1,2,\cdots$, we have
\begin{align*}
d(a_n,a_{n+p}) &\le d(a_n,a_{n+1})+ d(a_{n+1},a_{n+2})+ \cdots + d(a_{n+p-1},a_{n+p})\\
&\le k_n+ k_{n+1}+ \cdots + k_{n+p-1}\\
&\le k(\delta^{\frac{n}{3}-1} + \delta^{\frac{n+1}{3}-1} + \cdots + \delta^{\frac{n+p-1}{3}-1}) \\
&\le k \delta^{\frac{n}{3}-1} \dfrac{1-\delta^{\frac{p}{3}}}{1-\delta^{\frac{1}{3}}}.
\end{align*}

Thus, $d(a_n,a_{n+p}) \to 0$ as $n \to \infty$ for any $p=1,2,\cdots$. Therefore, $\{a_n\}$ is a Cauchy sequence and hence convergent as $M$ is complete. Let  $a_n \to a^* \in M$.  Now,
\begin{align*}
&d(a^*,Ta^*)\\
&\le d(a^*,a_n)+ d(a_n,Ta^*)\\
&\le d(a^*,a_n)+ d(Ta^*,Ta_{n-1})+ d(Ta_{n-1},Ta_{n-2})+ d(Ta_{n-2},Ta_n)+ d(Ta_n,Ta^*)\\
&\le d(a^*,a_n)+ \lambda (d(a^*,Ta_{n-1})+ d(a^*,Ta_{n})+ d(a_{n-1},Ta^*)+ d(a_{n-1},Ta_{n-2})\\
&+ d(a_{n-2},Ta_{n-1})+ d(a_{n-2},Ta_{n})+ d(a_{n},Ta_{n-2})+ d(a_{n},Ta^*))\\
&\le d(a^*,a_n)+ \lambda (d(a^*,a_{n})+ d(a^*,a_{n+1})+ d(a_{n-1},Ta^*)+ d(a_{n-1},a_{n-1})\\
&+ d(a_{n-2},a_{n})+ d(a_{n-2},a_{n+1})+ d(a_{n},a_{n-1})+ d(a_{n},Ta^*)).
\end{align*}

Tending $n$ to infinity, we get
\begin{align*}
(1-2\lambda) d(a^*,Ta^*) \le 0.
\end{align*}
This implies $Ta=a$, i.e., $a \in Fix(T)$.

If possible, suppose that $T$ has four distinct fixed points, say, $p,q,r,s \in M$. 
Then from (\ref{qcc}), we have 
\begin{align*}
     &d(p,q)+ d(q,r)+ d(r,s)+ d(s,p)  \le 2\lambda (d(p,q)+ d(q,r)+ d(r,s)+ d(s,p))\\
     \implies &(1- 2\lambda) (d(p,q)+ d(q,r)+ d(r,s)+ d(s,p)) \le 0,
\end{align*}
which contradicts the fact that $p,q,r,s$ all are distinct since $(1- 2\lambda)>0$.

Thus, $T$ can contain at most three fixed points. 
\end{proof}

Now, an example is provided to show the existence of a Chatterjea type perimetric contraction quadrilaterals mapping with three fixed points.  

\begin{example}\label{E55} 
Let $M=\{w,x,y,z \}$ and the metric $d$ be defined on $M$ by, $d(x,y)= d(y,z)= d(z,x)= 1$ and $d(x,w)= d(y,w)= d(z,w)= 10$. Let us define the mapping $T$ by $Tx=x, Ty=y, Tz=z, Tw=x$.

Then $(M,d)$ is a complete metric space and $T$ is a Chatterjea type perimetric contraction on quadrilaterals on $M$. Also, $T$ has no periodic points of prime period $2$ and $3$. Therefore by Theorem~\ref{T54}, $T$ has three fixed points and $Fix(T)= \{x,y,z\}$.
\end{example}
While the converse of Theorem ~\ref{T44} does not necessarily hold in general, a simple proof of a partial converse exists, which we exclude here. 
\begin{theorem}\label{T56}
Let $(M,d)$ be a complete metric space with at least four points, and let $T: M \to M$ be a Chatterjea type perimetric contraction on quadrilaterals. If $T$ possesses a fixed point, then  $T$ does not possess periodic points of prime period $3$. 
\end{theorem}



In Example~\ref{E55} we have seen that Chatterjea type perimetric contraction quadrilaterals may achieve more than one fixed point. Thus we can ensure the existence of a unique fixed point of such a mapping if an infinite complete metric space is considered which is shown in our next result:

\begin{proposition}\label{P57}
Let $(M,d)$ be a complete metric space, and let $T: M \to M$ be a Chatterjea type perimetric contraction on quadrilaterals. If $M$ contains infinitely many points such that the iteration sequence $m_0, m_1=Tm_0, m_2=Tm_1,\dots $, converges to a point $\xi$ with $\xi \neq m_i$; for all $i \in \mathbb{N} \cup \{0\}$, then $\xi$ is the unique fixed point of $T$.
\end{proposition}

\begin{proof}
Theorem~\ref{T54} assures that $T$ has a fixed point $\xi$. Let $\eta$ be another fixed point of $T$. Then $\eta \neq m_i$, for all $i \in \mathbb{N} \cup \{0\}$, otherwise we have $\xi = \eta$. Therefore, $\xi, \eta, m_i$ are all distinct, for all $i \in \mathbb{N} \cup \{0\}$.

Now, for all $i\in \mathbb{N}\cup \{0\}$, we have
\begin{align*}
&d(T\xi, Tm_i)+ d(Tm_i, T\eta)+ d(T\eta,Tm_{i+1})+ d(Tm_{i+1},T\xi)\\
&\le \lambda_i (d(\xi, Tm_i)+ d(\xi,Tm_{i+1})+ d(m_i, T\xi)+ d(m_i, T\eta))\\
&+ d(\eta, Tm_i)+ d(\eta, Tm_{i+1})+ d(m_{i+1}, T\eta)+ d(m_{i+1}, T\xi),
\end{align*}
which implies,
\begin{align*}
&d(\xi, m_{i+1})+ d(m_{i+1}, \eta)+ d(\eta,m_{i+2})+ d(m_{i+2},\xi)\\
&\le \lambda_i (d(\xi, m_{i+1})+ d(\xi,m_{i+2})+ d(m_i, \xi)+ d(m_i, \eta))\\
&+ d(\eta, m_{i+1})+ d(\eta, m_{i+2})+ d(m_{i+1}, \eta)+ d(m_{i+1}, \xi).
\end{align*}
Then $\lambda_i \le \lambda$ for all $i \in \mathbb{N}\cup \{0\}$. Now, letting $i \to \infty$, we get $\lambda_i \to 1$ - which contradicts (\ref{qcc}).
Therefore, $T$ has a unique fixed point.
\end{proof}

\begin{Acknowledgement}
The first author of the paper would like to thank the Council of Scientific and Industrial Research, Government of India, for supporting financially to carry out this work. The authors are also thankful to Subhadip Pal, NIT Durgapur, for his suggestions during the preparation of the manuscript.
\end{Acknowledgement}
%

\end{document}